\documentclass[a4paper]{amsart}
\usepackage{natbib}

\usepackage{amssymb}
\usepackage{amsmath}
\usepackage{amsthm}
\usepackage{graphicx}
\usepackage{tabularx}
\usepackage{multirow}
\usepackage{algorithm}
\usepackage{algorithmic}

\DeclareMathOperator{\Hom}{Hom}
\DeclareMathOperator{\Ann}{Ann}
\DeclareMathOperator{\Mon}{Mon}
\DeclareMathOperator{\LM}{LM}

\theoremstyle{plain}
\newtheorem{thm}{Theorem}[section]
\newtheorem{lem}[thm]{Lemma}
\newtheorem{prop}[thm]{Proposition}

\theoremstyle{definition}
\newtheorem{defn}[thm]{Definition}

\theoremstyle{remark}
\newtheorem{rem}[thm]{Remark}
\newtheorem{exmp}[thm]{Example}

\theoremstyle{remark}
\newtheorem{modification}[thm]{Modification}

\newcommand{\V}{{\mathbf V}}

\newcommand{\Q}{{\mathbb Q}}
\newcommand{\N}{{\mathbb N}}

\newcommand{\Z}{{\mathbb Z}}

\newcommand{\ideal}[1]{{\left\langle{#1}\right\rangle}}

\newcommand{\singular}{{\sc Singular}}

\newcommand{\proofend}[1]{\\[1ex]\parbox[b]{12cm}{#1}\hfill\ }

\begin{document}
\title{Normalization of Rings}

\author[G.-M. Greuel]{Gert-Martin Greuel}
\address {G.-M. Greuel - Department of Mathematics, University of Kaiserslautern \\ P.O. Box 3049 \\ (67653) Kaiserslautern, Germany}
\email{greuel@mathematik.uni-kl.de}
\author{Santiago Laplagne}
\address{S. Laplagne - Departamento de Matem\'atica, FCEN, Universidad de Buenos Aires \\ Ciudad Universitaria, Pabell\'on I - (C1428EGA) \\ Buenos Aires, Argentina}
\email{slaplagn@dm.uba.ar}
\author{Frank Seelisch}
\address{F. Seelisch - Department of Mathematics, University of Kaiserslautern \\ P.O. Box 3049 \\ (67653) Kaiserslautern, Germany}
\email{seelisch@mathematik.uni-kl.de}

\begin{abstract}
We present a new algorithm to compute the integral closure of a reduced Noetherian ring in its total ring of fractions. A modification, applicable in positive characteristic, where actually all computations are over the original ring, is also described. The new algorithm of this paper has been implemented in \singular, for localizations of affine rings with respect to arbitrary monomial orderings. Benchmark tests show that it is in general much faster than any other implementation of normalization algorithms known to us.
\end{abstract}

\date{January 26, 2010}

\maketitle

\section{Introduction}

Computing the normalization of a ring is a major tool in commutative algebra, with applications in algebraic geometry and singularity theory. The first general algorithms were proposed by \citet{Stolzenberg68} and \citet{Seidenberg70, Seidenberg75}. However, the tools involved, such as extensions of the ground field and addition of new indeterminates, make them unsuitable for most practical applications.

In recent years several new and more practicable algorithms using Groebner bases have been proposed. The basic approach, continuing the line of the works mentioned before, is to compute an increasing chain of rings from the original ring to its normalization. This is carried out in the works of \citet{Traverso86}, \citet{Vasconcelos91, Vasconcelos98}, \citet{BreVasc01}. To our knowledge none of these algorithms has been implemented and it remains unclear how practical they are. Also \citet{deJong98}, \citet{DGPJ} follow this path, applying as a new ingredient a criterion for normality due to \citet{GR}. In \citet{DGPJ} they report an effective implementation of their algorithm in \singular~\citep{GPS}. It became the standard algorithm for normalization in computer algebra systems, being now implemented also in Macaulay2 \citep{Macaulay2} and Magma \citep{Magma}. A good review on most of these algorithms can be found in \citet[Chapter 15]{SH}.

Another approach, presented in \citet{GT}, is to use Noether normalization, reduce the problem to the one dimensional case, and apply existing special algorithms for that case \citep{Ford87, Cohen93}. Unfortunately, we do not know of any implementation of these algorithms.

A more recent approach taken in \citet{LP} and \citet{SS} is to compute a decreasing chain of finitely generated modules over the original ring containing the normalization. Their algorithm works only in the case when the base field is of positive characteristic $p$, where they can use the Frobenius map. It has been implemented in Macaulay2 and \singular, and it turns out to be very fast for small $p$. However the computation of the Frobenius makes it impracticable when $p$ is large.

There are also very efficient methods for computing the normalization in some special cases. For example, for toric rings, one can apply fast combinatorial techniques, as explained in \cite{Bruns01}.

The algorithm we propose in this paper is a general algorithm and it is based on \citet{deJong98} and \citet{DGPJ}. In their algorithm, as we mentioned before, they construct an increasing chain of affine rings. They enlarge the rings by computing the endomorphism ring of a test ideal (see below), adding new variables for each module generator of the endomorphism ring and dividing out the relations among them. Then the algorithm is applied recursively to this new affine ring. Due to the increasing number of variables and relations this can produce a big slow-down in the performance of the algorithm already when the number of intermediate rings is 2 or 3. For a larger number, it usually makes the algorithm unusable, as the Groebner bases of the ideals of relations grow extensively. Our approach avoids the increasing complexity when enlarging the rings, benefiting from the finitely generated $A$-module structure of the normalization. We are able to do most computations over the original ring without adding new variables or relations.

The main new results of this paper are presented in Section \ref{section:original}.
In Section \ref{section:algorithm} we describe the algorithm and show, as an application, how the $\delta$-invariant of the ring can be computed. Section \ref{section:examples} contains several benchmark examples and a comparison with previously known algorithms, while Section \ref{section:extension} is devoted to an extension of the algorithm to non-global monomial orderings.

\section{Basic definitions and tools}

Let $A$ be a reduced Noetherian ring.$\!$\footnote{We assume that all rings are commutative with $1$ and that morphisms map $1$ to $1$.} The \emph{normalization} $\bar{A}$ of $A$ is the integral closure of $A$ in the total ring of fractions $Q(A)$, which is the localization of $A$ with respect to the non-zerodivisors on $A$. $A$ is called \emph{normal} if $A = \bar{A}$.

The \emph{conductor} of $A$ in $\bar{A}$ is $C = \{a \in Q(A) \mid a\bar{A} \subset A\} = \Ann_A(\bar{A} / A)$.

\begin{lem}\label{lem:fingenmod}
$\bar{A}$ is a finitely generated $A$-module if and only if $C$ contains a non-zerodivisor on $A$.
\end{lem}
\begin{proof}
If $p\in C$ is a non-zerodivisor then $\bar{A}\cong p\bar{A} \subset A$ is module-finite over $A$, since $A$ is Noetherian. Conversely, if $\bar{A}$ is module-finite over $A$ then any common multiple of the denominators of a finite set of generators is a non-zerodivisor on $A$ contained in $C$.
\end{proof}

We recall the Grauert and Remmert criterion of normality.

\begin{prop}
\label{prop:testideal}
Let $A$ be a Noetherian reduced ring and $J \subset A$ an ideal satisfying the following conditions:
\begin{enumerate}
\item $J$ contains a non-zerodivisor on $A$,
\item J is a radical ideal,
\item $N(A) \subset V(J)$, where
$$N(A) = \{P \subset A, \mbox{ prime ideal }\mid A_P \mbox{ is not normal}\}$$
is the non-normal locus of $A$.
\end{enumerate}
Then $A$ is normal if and only if $A \cong \Hom_A(J, J)$, via the canonical map which maps $a\in A$
to the multiplication by $a$.
\end{prop}

\begin{defn}
\label{defn:teststuff}
An ideal $J \subset A$ satisfying properties (1)$-$(3) is called a \emph{test ideal} (for the normalization) of $A$. A pair $(J, p)$ with $J$ a test ideal and $p\in J$ a non-zerodivisor on $A$ is called a \emph{test pair} for $A$.
\end{defn}
By Lemma \ref{lem:fingenmod}, test pairs exist if and only if $\bar{A}$ is module-finite over $A$. We can choose any radical ideal $J$ such that $p\in J\subseteq\sqrt{C}$.

Our algorithm computes the normalization of $A$ when a test pair for $A$ is known. If $A$ is a reduced, finitely generated $k$-algebra with $k$ a perfect field, then $C$ contains a non-zerodivisor which can be computed by using the Jacobian ideal (cf.~Lemma \ref{lem:jacobperf} and Remark \ref{rem:knotperf}). The same holds for localizations of such $k$-algebras w.r.t.~any monomial ordering. Indeed, our algorithm is slightly more general, working whenever the Jacobian ideal does not vanish.

If $A$ is not normal, we get a proper ring extension $A \subsetneq \Hom_A(J, J)=:A_1$.

If $A_1$ is not normal, which is checked by applying Proposition \ref{prop:testideal} to $A_1$, we obtain a new ring $A_2$ by that same proposition, which then has to be tested for normality, and so on. That is, we get a chain of inclusions of rings
$$A \subset A_1 \subset A_2 \subset \dots $$
(with $A_i = A[t_1, \dots, t_{s_i}]/I_i$, $I_i$ ideal, and natural maps $\psi_i: A \hookrightarrow A_i$).

If at some point, we get a normal ring $A_N$, then $A_N \cong \bar A$ by Lemma \ref{lem:abnorm}. 
This guarantees that if $\bar{A}$ is a finitely generated $A$-module, the chain will become stationary with $A_N$ normal, 
giving an algorithm to compute the normalization.

\vspace{0.4cm}

The fact which makes the whole algorithm practicable, is the isomorphism
$$\Hom_A(J, J)\cong 1/p\cdot(pJ:_A J),$$
allowing us to compute $\Hom_A(J, J)$ over $A$. This fact, not contained in \citet{deJong98}, was found by the first named author during the implementation of the algorithm in \singular\ and first published in \citet{DGPJ} (see also \citet[Lemma 3.6.1]{GP} and \citet{GT} for related statements). We shall prove a generalization of this isomorphism in Lemma \ref{lem:newlem}, which will be needed in the new algorithm.

The following lemma describes the $A$-algebra structure of $1/p\cdot(pJ:_A J)$. This will allow us to compute the normalization of $A$ recursively.

\begin{lem}
\label{lem:aalgisom}
Let $A$ be a reduced Noetherian ring, and $(J, p)$ a test pair for $A$. Let $\{u_0 = p, u_1, \dots, u_s\}$ be a system of generators for the A-module $pJ :_A J$. If $t_1, \dots, t_s$ denote new variables, then $t_j \mapsto u_j/p$, $1\leq j\leq s$, defines an isomorphism of $A$-algebras
$$A_1 := A[t_1, \dots, t_s]/{I_1} \stackrel{\cong}{\longrightarrow} \frac{1}{p} (pJ:_A J),$$
where $I_1$ is the kernel of the map $t_j \mapsto u_j/p$ from $A[t_1, \dots, t_s] \rightarrow \frac{1}{p} (pJ:_A J)$.
\end{lem}

See \citet[Lemma 3.6.7]{GP} for the computational aspects of this lemma.

\begin{exmp}
\label{exmp:x2-y3}
Let $I =\ideal{x^2 - y^3}\subset k[x, y]$ and $A=k[x,y]/I$. We take the test pair $(J, p)$, with $J := \ideal{x, y}_A$ (the radical of the singular locus of $A$) and $p := x$ (see Algorithm \ref{algorithm:testpair}). Then $pJ :_A J = \ideal{x, y^2}_A$ and $1/p\cdot(pJ:_A J) =  1/x\cdot\ideal{x, y^2}_A\cong A_1 := A[t]/I_1$ where $I_1 = \langle t^2-y, yt-x, y^2-xt \rangle_{A[t]}$. The isomorphism is given by $t \mapsto y^2/x$.
\end{exmp}

The following easy lemma gives a normalization criterion for ring extensions. It provides a convenient way to prove correctness of our normalization algorithm, or any modification, because it is independent of the intermediate steps.

\begin{lem}
\label{lem:abnorm}
Let $\psi: A \rightarrow B$ be a map between reduced Noetherian rings satisfying the following conditions:
\begin{enumerate}
\item $\psi$ is injective,
\item $\psi$ is finite,
\item \label{item:AB}$B$ is contained in $Q(\psi(A))$.
\end{enumerate}
Then $\psi$ induces isomorphisms $Q(A) \rightarrow Q(B)$ and
$\bar A \rightarrow \bar B$. In particular, if $B$ is integrally closed, then $\bar A$ is isomorphic to $B$.
\end{lem}

\begin{pf}
Since $A \hookrightarrow B$ is injective, so is $Q(A) \hookrightarrow Q(B)$ and hence $\bar{A} \hookrightarrow \bar{B}$. The isomorphism $Q(A) \rightarrow Q(B)$ is clear by \ref{item:AB}. The finiteness of $\psi$ implies that $B$ (and therefore $\bar B$) is integral over $A$. Since $\psi(A) \subseteq B \subseteq \bar B  \subseteq Q(B) = Q(\psi(A))$, we conclude that $\bar B$ is the normalization of $\psi(A)$, which immediately implies the isomorphism $\bar A \rightarrow \bar B$.
\end{pf}

\section{Computing over the original ring}
\label{section:original}

It has already been noticed by many authors (see for example the comments preceding Prop. 6.65 of \citet{Vasconcelos05}) that the chain of rings mentioned in last section, or similar constructions where the number of variables and relations increase in each step, behaves poorly in practice. (See also Remark \ref{rem:normalc}.)

There has been therefore a search for algorithms carrying out most of the computations  in the original ring. In \citet{Vasconcelos00}, the author proposes to use
$$B = \bigcup_{n \geq 1} \Hom_S(I^n, I^n),$$
where $S$ is a hypersurface ring over which $A$ is finite and birational and $I$ is the annihilator of the $S$-module $A/S$. However, as mentioned in that same paper, computing $B$ is still the hard part of the algorithm and there is no indication on how to do it.

In this section we show that a chain of ring as used in \citet{DGPJ} can be constructed doing most computations over the original ring. In this way we obtain an algorithm that is usually much faster in practice.

The purpose of this section is not only to show that the computations over the original ring are possible. The proofs which we provide show also how these computations can be done and thus prepare the algorithms presented in the next section.

We start with a generalization of the isomorphism from the previous section, expressing $\Hom_A(J, J)$ as an ideal quotient, to be used later. We formulate a more general version than needed. For a related statement see \citet{SH}.

\begin{lem} \label{lem:newlem}
Let $A$ be a reduced (not necessarily Noetherian) ring, $Q(A)$ its total ring of fractions, and $I, J$ two $A$-submodules of $Q(A)$. Assume that $I$ contains a non-zerodivisor $p$ on $A$.
\begin{enumerate}
\item The map
$$ \Phi:\Hom_A(I, J)\stackrel{\cong}{\longrightarrow}\frac{1}{p}(pJ:_{Q(A)}I)=J:_{Q(A)}I,\:\:\:\varphi\mapsto \frac{\varphi(p)}{p}, $$
is independent of the choice of $p$ and an isomorphism of $A$-modules.
\item If $J\subset A$ then
$$ pJ:_{Q(A)}I=pJ:_AI. $$
\end{enumerate}
\end{lem}
\begin{pf}
(1) Let $q \in I$ be another non-zerodivisor on $A$. Write $p=p_1/p_0$ and $q=q_1/q_0$, with $p_0,q_0$ non-zerodivisors contained in $A$ and $p_1,q_1\in A$.

Then $c:=p_0q_0\in A$ is a non-zerodivisor and $cp,cq\in A$ with $cpq\in I$. Since $\varphi\in \Hom_A(I,J)$ is $A$-linear, we can write
$$ cp\varphi(q)=\varphi(cpq)=cq\varphi(p), $$
whence $\varphi(p)/p=\varphi(q)/q$ in $Q(A)$, showing that $\Phi$ is independent of $p$.

Moreover, for any $f\in I$ we have
$$ \frac{\varphi(p)}{p}\cdot f=\frac{\varphi(cp)}{cp}\cdot f=\frac{\varphi(cpf)}{cp}=\frac{cp\varphi(f)}{cp}=\varphi(f)\in J, $$
in particular $\varphi(p)\cdot f\in pJ$. This shows that the image $\Phi(\varphi)$ is in $1/p\cdot(pJ:_{Q(A)}I)$. It also shows that $\varphi(p)=0 \Leftrightarrow \forall\:f\in I\:\:\varphi(f)=0\Leftrightarrow \varphi=0$ and hence that $\Phi$ is injective.

To see that $\Phi$ is surjective, let $q\in Q(A)$ satisfy $qI\subset J$. Denote by $m_q\in \Hom_A(I,J)$ the multiplication by $q$. Then $\Phi(m_q)=qp/p=q$ showing that $\Phi$ is surjective.

(2) During the proof of (1)~we have seen that
$$ pJ:_{Q(A)}I=\{\varphi(p)\:\mid\:\varphi\in \Hom_A(I, J)\}. $$
Hence, the claimed equality holds if and only if $\varphi(p)\in A$ for all $\varphi\in \Hom_A(I,J)$, which is clearly true if $J\subset A$.
\end{pf}
%

Recall the chain of extension rings from last section $A \subset A_1 \subset A_2 \subset \dots$ We have seen that we can compute the normalization of $A$ by computing the normalization of $A_i$ (Lemma \ref{lem:abnorm}). The next proposition explains how to obtain a test pair in $A_i$ from a given test pair in $A$. This is the only computation to be carried out in $A_i$.

\begin{prop}
\label{prop:testpair}
Let $A$ be a reduced Noetherian ring, $A' = A[t_1, \dots, t_s]/I'$ a finite extension ring, with natural inclusion $\psi: A \hookrightarrow A'$. If $(J,p)$ is a test pair for $A$ then setting $J' = \sqrt{\langle \psi(J)\rangle_{A'}}$, $(J',\psi(p))$ is a test pair for $A'$.
\end{prop}

\begin{pf}
Let $C$ be the conductor of $A$ in $Q(A)$ and $C'$ the conductor of $A'$ in $Q(A')$. We know that $N(A') = V(C')$, $N(A) = V(C)$ and $\psi(C) \subset C'$. Therefore $V(C') \subset V(C)$, which proves that $N(A') \subset N(\psi(A))$ since $\psi(A) \cong A$. We have $N(A) \subset V(J)$ by definition of $J$, and hence $N(A') \subset V(\psi(J))$. Now $\psi(p) \in J'$ is a non-zerodivisor on $A'$ and $(J', \psi(p))$ is a test pair for $A'$.
\end{pf}

\begin{exmp}
\label{exmp:J1}
Recall Example \ref{exmp:x2-y3}. We started with $A = k[x, y]/\langle x^2 - y^3 \rangle$ and test pair $(J, p) = (\langle x, y \rangle, x)$ and obtained $A_1 := A[t]/I_1 \cong 1/d_1 \cdot U_1$ where $I_1 = \langle t^2-y, yt-x, y^2-xt \rangle$, $d_1 = x$ and $U_1 = \langle x, y^2\rangle_A$.

We now compute $J_1 = \sqrt{\langle\psi_1(J)\rangle_{A_1}} = \sqrt{\langle x, y \rangle_{A_1}} = \langle x, y, t \rangle_{A_1} = \langle t \rangle_{A_1}$ (since $t^2 = y$ and $t^3 = x$ in $A_1$). Therefore $(\langle t \rangle, x)$ is a test pair for $A_1$.
\end{exmp}

For the remainder of this section, let $R$ be a Noetherian ring, $I \subset R$ a radical ideal and $A = R / I$.

We are mainly interested in $R = k[x_1, \dots, x_n]$ with $k$ a field (which the reader may assume in the following), or $R = k[x_1, x_2,\dots, x_n]_>$ with $>$ an arbitrary monomial ordering. However the proposed method works quite generally, whenever a test pair is known.


In the new algorithm, we will compute ideals $U_1, U_2, \dots, U_N$ of $A$ and non-zerodivi\-sors $d_i\in U_i, 1\leq i\leq N$, on $A$ such that
$$A \subset \frac{1}{d_1}U_1 \subset \frac{1}{d_2}U_2 \subset \dots \subset \frac{1}{d_N}U_N = \bar{A}.$$
From the construction we know that $1/d_i\cdot U_i$ is a finitely generated $R$-algebra and hence there is a surjection
$$ R_i := R[t_1,t_2,\ldots,t_{s_i}]\twoheadrightarrow \frac{1}{d_i}U_i,\;\;t_j\mapsto u_j,$$
where $\{d_i,u_1,\ldots, u_{s_i}\}$ is a set of $R$-module generators of $U_i$. If $I_i$ denotes the kernel of this map, we get a ring map
$$ \varphi_i:A_i:=R_i/I_i\stackrel{\cong}{\longrightarrow}\frac{1}{d_i}U_i\subset Q(A).$$

(Note that the definition of $I_i$ is now slightly different from the one given in \ref{lem:aalgisom}.)

\begin{exmp}
\label{exmp:phiJ1}
Carrying on with Example \ref{exmp:J1}, we compute $\varphi_1(J_1) = \varphi_1(\langle t \rangle)$.

Note that $\varphi_1(t) = y^2/x$. However the $A$-module $\langle y^2/x \rangle_A \subsetneq \varphi_1(\langle t \rangle)$. For example, we have seen that $y \in \langle t \rangle_{A_1}$ and clearly $\varphi_1(y) = yx / x$, but $yx \not\in \langle y^2 \rangle_A$.

This shows that in order to obtain $A$-module generators of $\varphi_i(J_i)$ it is not enough to compute the images of the generators of $J_i$. In Algorithm \ref{algorithm:hi} we will show how to compute the generators. In this example, it turns out that $\varphi_1(\langle t \rangle) = \langle yx / x, y^2/x \rangle$ as $A$-module.
\end{exmp}

Once we have computed a test pair $(J_i, p_i)$ in $A_i$, the next step is to compute the quotient $pJ_i :_{A_i} J_i$. The following theorem shows that this computation can be carried out in the original ring $A$.
\begin{thm}
\label{thm:jihi}
Let $A = R/I$, $A' = A[t_1, \dots, t_s]/I'$ a finite ring extension and maps $\psi: A \hookrightarrow A'$, $\varphi:A' \hookrightarrow Q(A)$. Let $(J, p)$ be a test pair for $A$ and $(J', p')$ a test pair for $A'$, with $p' = \psi(p)$. Let $U, H$ be ideals of $A$ and $d \in A$ such that $\displaystyle{\varphi(A') = \frac{1}{d}U}$ and $\displaystyle{\varphi(J') = \frac{1}{d}H}$.  Then
$$(p'J') :_{A'} J' = \frac{1}{d}(d p H :_A H).$$
\end{thm}
\begin{pf}
The proof is an easy consequence of Lemma \ref{lem:newlem}. Omitting $\varphi$ and $\psi$ in the following notations and applying Lemma \ref{lem:newlem} to $p\in J\subset A$ we get
$$(p'J') :_{A'} J' = (p'J') :_{Q(A)} J' = pH :_{Q(A)} H, $$
since $Q(A')=Q(A)$ and $J'=1/d\cdot H$.

On the other hand, we can apply Lemma \ref{lem:newlem} to $d p\in H\subset A$ and get
\proofend{
$$\frac{1}{d}(d pH:_A H)=\frac{1}{d}(d pH:_{Q(A)}H)=pH :_{Q(A)}H.$$
}
\end{pf}

Using Theorem \ref{thm:jihi} together with the previous results, once we have computed an intermediate ring $A_i$, we can compute $A_{i+1}$, the next ring in the chain. If $A_i=A_{i+1}$, we have finished and $A_i$ is the normalization of the original ring $A$, by Lemma \ref{lem:abnorm}. If not, we proceed by induction to compute the normalization.

We continue with the above example.

\begin{exmp}
\label{exmp:finish}
We have $p = d_1 = x$ and $H_1 = \langle xy, y^2 \rangle_A$. We compute $d_1 p H_1 :_A H_1 = x^2 \langle xy, y^2 \rangle :_A \langle xy, y^2 \rangle = \langle x^2, xy^2 \rangle$.

Then
$$\Hom_{A_1}(J_1, J_1) \cong \frac{1}{x^2}\langle x^2, xy^2 \rangle = \frac{1}{x} \langle x, y^2 \rangle.$$
This is equal to $A_1$. Therefore, the ring $A_1$ was already normal, and hence equal to the normalization of $A$.
\end{exmp}

\begin{modification}
\label{modification:prime}
We have seen that the only computation performed in $A_i$ is the radical of $\psi_i(J)$. However, when the characteristic of the base field is $q > 0$ it is possible to compute also this radical over the original ring. For this, we use the Frobenius map, as described in \citet{M}.

Let $G = \psi_i(J) \subset A_i$. By definition,
$$J_i = \sqrt{G} = \{f \in A_i \mid f^{\,m} \in G \mbox{ for some } m \in \N\}.$$
Mapping to $Q(A)$, we obtain
$$\varphi_i(J_i) = \left\{\tilde f/d_i\:\:\bigg|\:\:\tilde f \in U_i,\:\left(\tilde f/d_i\right)^m \in \varphi_i(G) \mbox{ for some } m \in \N\right\}=\bigcup_{m\geq 1}G_m,$$
where $G_m := \displaystyle{\left\{\tilde f/d_i\:\:\bigg|\:\:\tilde f \in U_i,\:\left(\tilde f/d_i\right)^m \in \varphi_i(G)\right\}}$. Then
$$d_iG_q = \{\tilde f \in U_i \mid \tilde f^{\,q} \in d_i^{\,q}\varphi_i(G)\}.$$

Now $d_i^{\,q}\varphi_i(G)$ is an ideal of $A$ and $d_iG_q$ is the so-called $q$-th root of $d_i^{\,q}\varphi_i(G)$. This ideal can be computed over $A$ using the Frobenius map (cf.~\citet{M}).

By iteratively computing the $q$-th root of the output, until no new polynomials are added, we obtain $\varphi_i(J_i)$ as desired.

\vspace{0.3cm}

Computing the radical in this way, we get another algorithm (in positive characteristic) which is similar to the one proposed in \citet{SS}.
In their algorithm they start with the inclusion $\bar{A} \subset \frac{1}{c}A$, where $c$ is an element of the conductor and compute a decreasing chain of $A$-modules
$$\frac{1}{c}A = \frac{1}{c}U'_0 \supset \frac{1}{c}U'_1 \supset \dots \supset \frac{1}{c}U'_N = \bar{A}.$$

In our algorithm we compute an increasing chain
$$A \subset \frac{1}{d_1}U_1 \subset \dots \subset \frac{1}{d_N}U_N = \bar{A}.$$

The most difficult computational task for both algorithms is the Frobenius map. However, in our algorithm we start with a small denominator $d_1$ and therefore the computations might be in some cases easier. This modification has not yet been tested.
\end{modification}

\section{Algorithms and application}
\label{section:algorithm}
We describe the algorithm in general terms. Since we compute an increasing sequence of subrings of the integral closure the algorithm terminates, for a Noetherian ring $A$, if and only if $\bar{A}$ is a finitely generated $A$-module. By Lemma \ref{lem:fingenmod} this is equivalent to the existence of a test pair. We now deal with the problem of constructing an initial test pair.

\begin{lem}
\label{lem:jacobperf}
Let $k$ be a perfect field, and $A=k[x_1,x_2,\ldots,x_n]/I$ with $I=\ideal{f_1,f_2,\ldots,f_t}$ a reduced equidimensional ring of dimension $r$. Let $M$ be the Jacobian ideal of $I$, that is, the ideal in $A$ generated by the images of the $(n-r)\times(n-r)$-minors of the Jacobian matrix $(\partial f_i/\partial x_j)_{i,j}$. Then $M$ is contained in the conductor of $A$ and contains a non-zerodivisor on $A$.
\end{lem}
\begin{pf}
Let $I=P_1\cap P_2\cap\ldots\cap P_s$ with $P_1,P_2,\ldots, P_s$ the minimal associated primes of $I$. Since $A$ is equidimensional, $\dim(A)=\:\:$height$(P_i)=r$ for $1\leq i\leq s$. Hence, the image of $M$ in $A_i=k[x_1,x_2,\ldots,x_n]/P_i$ is contained in the Jacobian ideal $M_i$ of $P_i$. By the Lipman-Sathaye theorem (cf.~\citet{SH} and \citet[Remark 1.5]{SS}) $M_i$ and hence $M$ is contained in the conductor of $A_i$. Since $\bar{A}=\bar{A_1}\oplus\bar{A_2}\oplus\cdots\oplus\bar{A_s}$, $M$ is then also in the conductor of $A$. Moreover, the image of $M$ in $A_i$ is not zero since $A_i$ is reduced. This follows from the Jacobian criterion and by Serre's condition for reducedness (cf.~\citet[Section 5.7]{GP}). As a consequence, $M$ is not contained in the union of the minimal associated primes of $A$ and hence contains a non-zerodivisor on $A$.
\end{pf}

Note that both, the Lipman-Sathaye theorem and the Jacobian criterion, require $k$ to be perfect.

The ideal $J:=\sqrt{M}$ from last lemma can be used as an initial test ideal. To construct a test pair, we need to find in addition a non-zerodivisor of $A$ in $J$. An element $p\in A$ is a non-zerodivisor if and only if $0:_A\ideal{p}=0$, hence the non-zerodivisor test is effective. However, it is not sufficient to apply the test to the generators of $J$. (E.g., if $I=\ideal{xy}$, the polynomials $x,y$ generate $J$ and are zerodivisors on $A$, but $x+y$ is not.) Since we cannot test all elements of $J$ there seems to be a problem to find a test pair if $I$ is not prime.
We address this problem as well as the perfectness and the equidimensionality assumptions in Remark \ref{rem:knotperf}.

We first describe in Algorithm \ref{algorithm:testpair} how to compute the initial test pair $(J, p)$ in $A$, assuming that we are able to find a non-zerodivisor.

\begin{rem} \label{rem:knotperfmain}
Only for this step we need the assumption that $R=k[x_1,x_2,\ldots,x_n]_>$ with $k$ perfect and that $I$ is equidimensional. All further steps do not require this assumption.\\
If, by whatever means, an initial test pair $(J, p)$ for $A$ is known, we can start with the computation of $U_1$ and then all further steps are correct, and the loop terminates with the computation of $\bar{A}$. Hence, for any reduced ring $A=R/I$ with given test pair $(J, p)$, the algorithm is effective when Gr{\"o}bner bases, ideal quotients, and radicals can be computed in rings of the form $R[t_1,t_2,\ldots,t_s]$.
\end{rem}

\begin{algorithm}
\caption{Initial test pair $(J, p)$}
\label{algorithm:testpair}
\begin{algorithmic}
\REQUIRE $I\subset R$, an equidimensional radical ideal, with $R=k[x_1,x_2,\ldots,x_n]_>$ and $k$ a perfect field.
\ENSURE $(J, p)$ a test pair for $A := R/I$.
\STATE $r := \dim(I)$
\STATE $M' :=$ the Jacobian ideal of $I$, i.e., the ideal in $R$ generated by the
\STATE \hspace*{4ex}$(n-r)\times(n-r)$-minors of the Jacobian matrix of $I$
\STATE $M :=$ the image of $M'$ in $A$
\STATE $J := \sqrt{M} \subset A$
\STATE \label{step:nzd}choose $p \in J$ such that $p$ is a non-zerodivisor on $A$
\RETURN $(J, p)$
\end{algorithmic}
\end{algorithm}

We now explain how to perform some auxiliary tasks, that will be needed in the main algorithm.

We have seen in the previous section that if $A = R/I$ and $A' = R[t_1, \dots, t_n]/I'$ a finite extension ring with $I \subset I'$, then there exist a non-zerodivisor $d \in A$, an ideal $U \subset A$ and a map $\varphi: A' \rightarrow 1/d \cdot U$ such that $A' \stackrel{\cong}{\longrightarrow} 1/d \cdot U$. For computations, we need to know how to move from one representation to the other.

\begin{rem} \label{rem:compphi}
If we know $d$ and generators $\{d, u_1, \dots, u_s\}$ of $U$, we can explicitly compute $\varphi(q)$ for any $q \in A'$. Let $\tilde q \in R'$ be a representative, and substitute all the variables $t_j$ in $\tilde q$ by the corresponding fraction $u_j / d$. This results in an element $f / d^{\,e} \in Q(A)$ for some $f \in A$ and $e \in \Z_{\geq 0}$. Now we need to find $f' \in A$ such that $f/d^{\,e} = f' / d$ in $Q(A)$, which is equivalent to $f = f'd^{\,e-1} + g$ in $R$, with $g \in I$. We can find $f'$ by solving the (extended) ideal membership problem $f \in I + \langle d^{\,e-1}\rangle$ in $R$, e.g.~by using the \singular\ command \texttt{lift}, cf.~\citet[Example 1.8.2]{GP}.
\end{rem}

We will need also to compute $A$-module generators of ideals $J' \subset A'$ given by generators in $A'$. It is clear that for any such $J'$ there exist an ideal $H \subset A$ such that $\displaystyle{\varphi(J') = 1/d\cdot H}$. So the problem is equivalent to finding elements $h_1, \dots, h_l$ in $A$ that generate $H$ as an $A$-ideal. In Algorithm \ref{algorithm:hi} we explain how do it.

\begin{algorithm}
\caption{$A$-module generators}
\label{algorithm:hi}
\begin{algorithmic}
\REQUIRE $A = R/I$, with $R = k[x_1, \dots, x_n]$ and $I \subset R$ ideal; $A' = R'/I'$ a ring extension of $A$, with $R' = R[t_1, \dots, t_s]$ and $I' \subset R'$ an ideal; $d \in A'$ a non-zerodivisor and $U' = \langle u_0 = d, u_1, \dots, u_s \rangle_A$ such that $A' \cong \frac{1}{d}U$, with map $\varphi: A' \stackrel{\cong}{\longrightarrow} \frac{1}{d}U$; $J' = \langle f_1, \dots, f_m \rangle_{A'}$, an ideal of $A'$.
\ENSURE $H = \langle h_1, \dots, h_l \rangle_A$ such that $\displaystyle{\varphi(J') = 1/d\cdot H}$
\FOR{$j = 1, \dots, m$}
\STATE compute $h_j$ such that $\varphi_i(f_j) = h_j / d$ (cf. Remark \ref{rem:compphi})
\ENDFOR
\STATE set $S = \{h_1, \dots, h_m\}$
\FOR{$j=1, \dots, m$; $k = 1, \dots, s$}
\STATE compute $h_{j,k} \in A$ such that
$h_{j,k}/d = u_k/d\cdot h_j/d$ in $Q(A)$ (again by Remark \ref{rem:compphi})
\IF{$h_{j,k} \not\in \langle S \rangle_A$}
\STATE $S = S \cup \{h_{j,k}\}$
\ENDIF
\ENDFOR
\RETURN{$H := \langle S \rangle$}
\end{algorithmic}
\end{algorithm}

\begin{lem}
\label{lem:uh}
Let $A = R/I$, with $R = k[x_1, \dots, x_n]$ and $I \subset R$ ideal; $A' = R'/I'$ a ring extension of $A$, with $R' = R[t_1, \dots, t_s]$ and $I' \subset R'$ an ideal; $d \in A'$ a non-zerodivisor and $U' = \langle u_0 = d, u_1, \dots, u_s \rangle_A$ such that $A' \cong \frac{1}{d}U$, with map $\varphi: A' \stackrel{\cong}{\longrightarrow} \frac{1}{d}U$; $J' = \langle f_1, \dots, f_m \rangle_{A'}$, an ideal of $A'$. The output ideal $H = \langle h_1, \dots, h_l \rangle_A$ of Algorithm \ref{algorithm:hi} satisfies $\varphi(J') = 1/d \cdot H$.
\end{lem}

\begin{pf}
This follows since the $A$-module $\ideal{1=u_0/d, u_1/d_i, \ldots, u_s/d}_A=\varphi(A')$ and the $A'$-module $\ideal{h_1/d,h_2/d,\ldots,h_m/d}_{A'} =\varphi(J')$ ($h_1, \dots, h_m$ as in the algorithm). Therefore the products $u_k/d \cdot h_j/d$, $0 \leq k \leq s, 1 \leq j \leq m$, generate $\varphi(J')$ as $A$-module. Hence $\{h_j\:\mid\:1\leq j\leq l\}$ generates $H$ as $A$-module, or equivalently as $A$-ideal.
\end{pf}

\begin{exmp}
\label{exmp:generators}
We apply the algorithm to compute the $A$-module generators of $\varphi_1(J_1)$ from Example \ref{exmp:phiJ1}. Recall that $J_1 = \langle t \rangle_{A_1}$, $U_1 = \langle x, y^2\rangle_A$ and $d = x$. We start with $h_1 = \varphi_1(t) = y^2/x$ and $S = \{ h_1 \}$. In the first step, we compute $x/x \cdot y^2/x = y^2/x$, therefore $h_{1,0} = y^2$. Since $y^2 \in \langle y^2 \rangle$, we do not do anything. In the second step we compute $y^2/x \cdot y^2/x = y^4/x^2 = x^2y / x^2 = xy / x$, therefore $h_{1,1} = xy$. Since $xy \not\in \langle y^2 \rangle$, we add it to $S$. We finish with $H = \langle xy, y^2\rangle$, as mentioned in Example \ref{exmp:phiJ1}.
\end{exmp}

We are now ready to present in Algorithm \ref{algorithm:normalization} the main algorithm to compute the normalization.

Termination follows from Lemma \ref{lem:fingenmod} and the discussion after Definition \ref{defn:teststuff}, correctness follows from Lemma \ref{lem:abnorm}.

\begin{algorithm}
\caption{Normalization of $R/I$}
\label{algorithm:normalization}
\begin{algorithmic}
\REQUIRE $I\subset R$, an equidimensional radical ideal
\ENSURE generators of an ideal $U \subset R$, and $d \in R$ such that $\overline{A}=\displaystyle{\frac{1}{d} U \subset Q(A)}$,
\STATE \hspace*{8ex}with $A := R/I$.
\STATE compute $(J, p)$, an initial test ideal
\STATE \label{step:U1}$U_1 := (pJ :_A J) \subset A$
\STATE $d_1 := p$
\IF {$\ideal{d_1} = U_1$}
  \RETURN $(\ideal{1}, 1)$
\ENDIF
\STATE $i := 1$
\LOOP
\STATE write $U_i = \ideal{d_i, u^{(i)}_{1}, u^{(i)}_{2}, \dots, u^{(i)}_{s}}_A$
\STATE set $R_i := R[t_1, \dots, t_s]$, $\pi_i: R_i \rightarrow \frac{1}{d_i}U_i \subset \frac{1}{d_i}A$ the map $t_j \mapsto u_j^{(i)}/d_i$
\STATE $I_i := \ker(\pi_i)$ (cf. Lemma \ref{lem:aalgisom})
\STATE set $A_i = R_i / I_i$
\STATE $J_i := \sqrt{\psi_i(J)} \subset A_i$, with $\psi_i: A \hookrightarrow A_i$
\STATE compute $\{f_1, \dots, f_k\} \subset A$ such that $H_i := \ideal{f_1, f_2, \dots, f_k}_A = d_i \varphi_i(J_i)$,
\STATE \hspace*{4ex}with $\varphi_i: A_i \stackrel{\cong}{\longrightarrow} \frac{1}{d_i}U_i$ (cf. Lemma \ref{lem:uh})
\STATE compute generators of $U_{i+1} := (p d_i H_i)\: :_A H_i$
\IF {$d_iU_i \subset U_{i+1}$}
  \RETURN $(U_i, d_i)$
\ENDIF
\STATE $d_{i+1} := p d_i$
\STATE $i := i+1$
\ENDLOOP
\end{algorithmic}
\end{algorithm}

\begin{rem} \label{rem:knotperf}
Let us comment on some variations and generalizations of Algorithm \ref{algorithm:normalization}. For this
let $k$ be {\it any} field, $R=k[x_1,x_2,\ldots,x_n]_>,$ and $I\subset R$ a radical ideal.\\
(1) If $I$ is not (or not known to be) equidimensional we can start with an algorithm to compute the minimal associated primes (cf.~\citet[Algorithm 4.3.4, Algorithm 4.4.3]{GP}) or the equidimensional parts (cf.~\citet[Algorithm 4.4.9]{GP}) of $I$, where the latter is often faster. The corresponding ideals $I_1, I_2,\ldots,I_r$ are equidimensional and we have $\overline{R/I}\cong \overline{R/I_1}\oplus\overline{R/I_2}\oplus\cdots\oplus\overline{R/I_r}$. Hence the problem is reduced to the case of $I$ being prime or equidimensional.\\
(2) Now let $I$ be equidimensional and $M$ the Jacobian ideal. Since regular rings are normal, it follows from the Jacobian criterion that $N(R/I)\subset V(M)$. Let us assume that $M\neq 0$ and choose $p\in M\setminus\{0\}$.\\
a) If $I_1:=I:_R\ideal{p}\subset I$ then $p$ is a non-zerodivisor on $A$ and $J=\sqrt{M}$ is a test ideal. This is always the case if $I$ is prime.\\
b) If $I_1\not\subset I$ we compute $I_2:=I:_R I_1$ and obtain a splitting $I=I_1\cap I_2$ (cf.~\citet[Lemma 1.8.14(3)]{GP}) and $\overline{R/I}\cong \overline{R/I_1}\oplus\overline{R/I_2}$. Hence we can continue with the ideals $I_1$ and $I_2$ separately which have both fewer minimal associated primes than $I$. Consequently, after finitely many splittings, the corresponding ideal is prime or we have found a non-zerodivisor. This provides us with test ideals as in case a).\\
(3) The above arguments show that (even if $k$ is not perfect) Algorithm \ref{algorithm:normalization} works for prime ideals if and only if the Jacobian ideal $M$ is not zero. This is always the case for $k$ perfect. However, if $k$ is not perfect, $M=0$ may occur. For example, consider $k=(\mathbb{Z}/q)(t)$ with $q$ a prime number, and $I=\ideal{x^q+y^q+t}\subset k[x,y]$. For a method to compute a non-zero element in the conductor of $R/I$ if $I$ is prime and if $Q(R/I)$ is separable over $k$, see \citet[Exercise 12.12]{SH}.
\end{rem}

\subsection{The $\delta$-invariant}

As an application of the normalization algorithm we show how to compute the $\delta$-invariant of $A = k[x_1, x_2,\dots, x_n]_>/I$, a reduced Noetherian $k$-algebra,
$$\delta(A) := \dim_k(\bar{A} / A).$$
$\delta(A)$ may be infinite but it is finite if the algebraic variety $V(I)$ defined by $I$ has isolated non-normal points, e.g.~for reduced curves, i.e.~$\dim(A)=1$. In this case, $\delta$ is important as it is the difference between the arithmetic and the geometric genus of a curve. Moreover, the $\delta$-invariant is one of the most important numerical invariants for curve singularities (cf.~\citet{CGL}), that is, for 1-dimensional complete local rings $A$. The extension of our algorithm to non-global orderings in Section \ref{section:extension} has the immediate consequence that it allows to compute $\delta$ for affine rings as well as for local rings of singularities, noting that $\delta(k[x_1,x_2,\ldots,x_n]_>/I)=\delta(k[[x_1,x_2,\ldots,x_n]]/I)$ if $>$ is a local ordering.

\begin{lem}
\label{lem:delta}
Let $R$ be a reduced Noetherian ring, $I \subset R$ be a radical ideal, and $I = P_1 \cap \dots \cap P_r$ its prime decomposition. Write $I = I_1 \cap \dots \cap I_s$, where $I_i = \bigcap_{j \in N_i} P_j$ and $\{N_1, \dots, N_s\}$ is a partition of $\{1, \dots, r\}$. Let $U_i$, $d_i$ be the output of the normalization algorithm for $A_i = R / I_i$. Then
\begin{enumerate}
\item $\delta(A_i) = \dim_k(U_i / d_iU_i),\:\:\:1\leq i\leq s$,
\item \label{item:deltaRI} $\displaystyle{\delta(R / I) = \sum_{i=1}^s \delta(A_i) + \sum_{i = 1}^{s - 1} \dim_k(R/(I + I^{(i)}))},$
where $I^{(i)} = I_{i+1} \cap \dots \cap I_s$.
\end{enumerate}
In particular $\delta(R/I) < \infty$ iff every summand on the right hand side of (\ref{item:deltaRI})~is finite.
\end{lem}

\begin{pf}
This follows by induction on $s$, and by repeatedly applying the following sequence of inclusions for $s = 2$, i.e. $I = I_1 \cap I_2$,
$$R / I \hookrightarrow R/I_1 \oplus R/I_2 \hookrightarrow \overline{R/I_1} \oplus \overline{R/I_2} \cong \overline{R / I},$$
and the exact sequence
$$0 \rightarrow R/I \rightarrow R/I_1 \oplus R/I_2 \rightarrow R/(I_1 + I_2) \rightarrow 0.$$
\end{pf}
Note that $\dim_k R/(I_i + I^{(i)})$ can be computed from a standard basis of $I_i + I^{(i)}$ and $\dim_k(U_i / d_i U_i)$ from a standard basis of a presentation matrix of $U / d_iU_i$ via \texttt{modulo} (cf.~\citet[\singular\ Example 2.1.26]{GP}). An algorithm to compute $\delta$ is also implemented in \singular~\citep{GPS}.

\section{Examples and comparisons}
\label{section:examples}
\begin{table}
\caption{Timings}
\label{table:timings}
\begin{tabularx}{12cm}{|r|r!{\vrule width 1pt}X|r!{\vrule width 1pt}r|r|r|}
\noalign{\hrule height 1pt}
\multirow{2}{*}{No.} & \multirow{2}{*}{char} & \multicolumn{2}{c!{\vrule width 1pt}}{\texttt{normal} data} & \multicolumn{3}{c|}{seconds} \\ \cline{3-7}
& & non-zerodivisor & steps & \texttt{normal} & \texttt{normalP} & \texttt{normalC} \\ \noalign{\hrule height 1pt}
1 & 0 & $y$ & 7 & 0 & - & 72 \\ \hline
1 & 2 & $y$ & 7 & 0 & 0 & 0 \\ \hline
1 & 5 & $y$ & 7 & 1 & 73 & 0 \\ \hline
1 & 11 & $x-2y$ & 7 & 1& 12 & $*$ \\ \hline
1 & 32003 & $y$ & 7 & 0 & $*$ & 1 \\ \noalign{\hrule height 1pt}
2 & 0 & $y$ & 7 & 1 & - & $*$ \\ \hline
2 & 3 & $y$ & 8 & 0 & 0 & 3 \\ \hline
2 & 13 & $y$ & 7 & 0 & $*$ & 10 \\ \hline
2 & 32003 & $y$ & 7 & 0 & $*$ & 10 \\ \noalign{\hrule height 1pt}
3 & 0 & $y$ & 6 & 2 & - & $*$ \\ \hline
3 & 2 & $y$ & 13 & 1 & 0 & $*$ \\ \hline
3 & 5 & $y$ & 6 & 1 & 7 & $*$ \\ \hline
3 & 11 & $x + 4y$ & 6 & 1 & $*$ & $*$ \\ \hline
3 & 32003 & $y$ & 6 & 1 & $*$ & $*$ \\ \noalign{\hrule height 1pt}
4 & 0 & $2x^2y-y^3+y$ & 1 & 0 & - & 0 \\ \hline
4 & 5 & $x^2y+2y^3-2y$ & 1 & 0 & 3 & 0 \\ \hline
4 & 11 & $x^2y+5y^3-5y$ & 1 & 0 & $*$ & 0 \\ \hline
4 & 32003 & $x^2y+16001y^3- 16001y$ & 1 & 0 & $*$ & 0 \\ \noalign{\hrule height 1pt}
5 & 0 & $y$ & 1 & 0 & - & $0$ \\  \hline
5 & 5 & $x^3y+xy$ & 3 & 1 & $*$ & $*$ \\  \hline
5 & 11 & $y$ & 1 & 0 & $0$ & 0 \\  \hline
5 & 32003 & $y$ & 1 & 1 & $*$ & 0 \\  \noalign{\hrule height 1pt}
6 & 2 & $v$ & 2 & 6 & 24 & 172 \\ \noalign{\hrule height 1pt}
7 & 0 & $y$ & 6 & 12 & - & $582$ \\ \hline
7 & 2 & $y$ & 6 & 11 & 0 & $35$ \\ \hline
7 & 5 & $y$ & 6 & 12 & $3$ & $358$ \\ \hline
7 & 11 & $y$ & 6 & 11 & $43$ & $503$ \\ \hline
7 & 32003 & $y$ & 6 & 11 & $*$ & $617$ \\ \noalign{\hrule height 1pt}
\end{tabularx}
\end{table}

In Table \ref{table:timings} we see a comparison of the implementations in \singular\ of the new algorithm \texttt{normal} and other existing algorithms. \texttt{normalC} is an implementation based on the algorithm \citet{DGPJ} (see also \citet[Section 3.6]{GP}) and \texttt{normalP} is an implementation of the algorithm of \citet{LP}, \citet{SS} for positive characteristic. All these implementations are now available in the \singular\ library \texttt{normal.lib} \citep{GLP}. Computations were performed on a compute server running a 1.60GHz Dual AMD Opteron 242 with 8GB ram. \\
$*$ indicates that the algorithm had not finished after 20 minutes,\\
- indicates that the algorithm is not applicable (i.e., using \texttt{normalP} in characteristic 0).

We try several examples over the fields $k = \Q$ and $k = \Z_p, p\in\{2,3,5,11,13,32003\}$, when the ideal is prime in the corresponding ring. We see that the new algorithm is extremely fast compared to the other algorithms. Only the algorithm \texttt{normalP} is sometimes faster for very small characteristic.

In columns 3 and 4 we give additional information on how the new algorithm works. The column ``non-zerodivisor'' indicates which non-zerodivisor is chosen. The column ``steps'' indicates how many loop steps are needed to compute the normalization. We see that our new algorithm performs well compared to the classic algorithm especially when the number of steps needed is large.

We use the following examples:
\begin{itemize}
\item $I_1 = \langle (x-y)x(y+x^2)^3-y^3(x^3+xy-y^2) \rangle \subset k[x, y]$,
\item $I_2 = \langle 55x^8+66y^2x^9+837x^2y^6-75y^4x^2-70y^6-97y^7x^2 \rangle \subset k[x, y]$,
\item $I_3 = \langle y^9+y^8x+y^8+y^5+y^4x+y^3x^2+y^2x^3+yx^8+x^9 \rangle \subset k[x, y]$,
\item $I_4 = \langle (x^2+y^2-1)^3 +27x^2y^2 \rangle \subset k[x, y]$,
\item $I_5 = \langle -x^{10}+x^8y^2-x^6y^4-x^2y^8+2y^{10}-x^8+2x^6y^2+x^4y^4-x^2y^6-y^8+2x^6-x^4y^2+x^2y^4+2x^4+2x 2y^2-y^4-x^2+y^2-1
    \rangle \subset k[x, y] $,
\item $I_6 = \langle z^3+zyx+y^3x^2+y^2x^3, uyx+z^2,uz+z+y^2x+yx^2, u^2+u+zy+zx, v^3+vux+vz^2+vzyx+vzx+uz^3+uz^2y+z^3+z^2yx^2 \rangle \subset k[x, y, z, u, v]$.
\item $I_7 = \langle x^2+zw, y^3+xwt, xw^3+z^3t+ywt^2, y^2w^4-xy^2z^2t-w^3t^3 \rangle \subset k[x, y, z, w, t]$.
\end{itemize}

\begin{rem} \label{rem:normalc} As mentioned before, the main drawback of the algorithm \citet{DGPJ} is the increasing complexity of the new rings that are constructed. A direct implementation of the algorithm turns out to be so slow that it does not even finish for most of the examples analyzed in this paper (after 1 hour).
For example, in the second example ($I_2$) over $\Z_3$, the fifth ring constructed in the chain has 12 variables and 76 generators for the ideal of relations. The sixth ring could not be computed using this direct approach.

A partial solution to this problem, used in implementations, is to eliminate as far as possible redundant variables, that is, variables than can be expressed in term of the others through the relations in the ring. This is what is done in \texttt{normalC}, and it is sometimes a good improvement. However detecting the redundant variables becomes more and more difficult as the relations get more and more complex, adding a new expensive task to the computation, that does not always succeeds in detecting all the relations.

The algorithm proposed in this paper avoids this problem in a natural way.
\end{rem}

We have also compared our implementation with the normalization procedures in Macaulay2
(they use the algorithms \cite{DGPJ} and \cite{SS})
and in Magma (they say that they use \cite{DGPJ} for the
general case; however it seems to work only in characteristic 0 and
the code is not accessible). Our new algorithm is always faster and succeeds
where the other implementations do not finish. We do not know implementations
in other computer algebra systems.

\section{Extension to non-global orderings}
\label{section:extension}

In this section, let $>$ be any monomial ordering on the set $\Mon(x_1, \dots, x_n)$ of monomials in $x = (x_1, \dots, x_n)$. That is, $>$ is a total ordering which satisfies
$$\forall \alpha, \beta, \gamma \in \Z_+^n\:\:\:x^{\alpha} > x^{\beta} \Rightarrow x^{\alpha + \gamma} > x^{\beta + \gamma},$$
but we do not require that $>$ is a well ordering. The main reference for this section is \citet{GP} where the theory of standard basis for such monomial orderings was developed.

We consider the multiplicatively closed set
$$S_> := \{u \in k[x] \smallsetminus \{0\} \mid \LM(u) = 1\},$$
where $\LM$ denotes the leading monomial. The localization of $k[x]$ w.r.t.~$S_>$ is denoted as
$$k[x]_>:= S_>^{-1}k[x] = \left\{\frac{f}{u} \:\:\bigg|\:\: f, u \in k[x], \LM(u) = 1\right\}.$$
It is shown in \citet[Section 1.5]{GP} that $k[x]_>$ is a regular Noetherian ring satisfying
$$k[x] \subset k[x]_> \subset k[x]_{\langle x \rangle},$$
where $k[x]_{\langle x \rangle}$ denotes the localization of $k[x]$ w.r.t.~the maximal ideal $\langle x_1, \dots, x_n \rangle = \langle x \rangle$. Note that
\begin{itemize}
\item $k[x]_> = k[x]\:\:\Leftrightarrow\:\: >$ is global (i.e. $x_i > 1$, $i = 1, \dots, n$), and
\item $k[x]_> = k[x]_{\langle x \rangle}\:\:\Leftrightarrow\:\: >$ is local (i.e. $x_i < 1$, $i = 1, \dots, n$).
\end{itemize}

In applications, in particular in connection with elimination in local rings, we need also \emph{mixed} orderings, where some of the variables are greater than and others smaller than $1$. An important case is the product ordering $>\:=\:(>_1, >_2)$ on $\Mon(x_1,\dots, x_n, y_1,\dots, y_m)$
where $>_1$ is global on
$\Mon(x_1,x_2,\dots, x_n)$ and $>_2$ is arbitrary on $\Mon(y_1,y_2,\dots, y_m)$. Then
$$k[x, y]_> = (k[y]_{>_2})[x] = k[y]_{>_2} \otimes_k k[x],$$
(cf.~\citet[Examples 1.5.3]{GP}), which will be used in the extension of our algorithm to non-global orderings.

We now show that for any monomial ordering $>$ and any radical ideal $I \subset k[x]_>$, the normalization of the ring $k[x]_>/I$ is a finitely generated $k[x]_>/I$-module and how to extend Algorithm \ref{algorithm:normalization} from Section \ref{section:algorithm} to this general situation.

%
%
For any ideal $I \subset k[x]_>$ we have $I = I'k[x]_>$, with $I' = I \cap k[x]$. Let $(k[x]/I')_>$ (resp.~$(\overline{k[x]/I'})_>$) denote the localization w.r.t.~the image of $S_>$ in $k[x]/I'$ (resp.~in $\overline{k[x]/I'}$). We have $k[x]_>/I \cong (k[x]/I')_>$.

\begin{lem}
\label{lem:fingen}
With the above notations, we have an isomorphism
$$\overline{k[x]_>/I} \cong (\overline{k[x]/I'})_>$$
of $k[x]_>$-algebras. In particular, $\overline{k[x]_>/I}$ is a finitely generated $k[x]_>/I$-module.

Moreover, let $\overline{k[x]/I'} \cong k[x, t]/H$ as $k[x]$-algebras with new variables $t = (t_1, \dots, t_s)$ and $H$ an ideal in $k[x, t]$. Then
$$\overline{k[x]_>/I}\:\:\cong\:\:(k[x]_>)[t]\:/\:H(k[x]_>)[t].$$
\end{lem}

\begin{pf}
The first statement follows immediately from the well-known fact that localization commutes with normalization. Since $\overline{k[x]/I'}$ is module-finite over $k[x]/I'$ the same holds for the localization $(\overline{k[x]/I'})_>$ over $(k[x]/I')_>$. The last statement follows since the image of $S_>$ in $k[x,t]$ localizes $k[x, t]$ only w.r.t.~the $x$ variables.
\end{pf}

\begin{rem}
Let $f_1, f_2, \dots, f_s \in k[x]$ generate $I = \ideal{f_1, f_2, \dots, f_s}k[x]_>$ and let $I'$ denote the ideal generated by $f_1,f_2,\ldots,f_s$ in $k[x]$. We can compute $\overline{k[x]_>/I}$ in two different ways.

The first method is to compute a test ideal $J$ and $\Hom_{k[x]_>/I}(J, J)$ in the same manner as described in the previous sections, just w.r.t.~the ordering $>$, i.e. in $k[x]_>$. When adding new variables $t_i$ (corresponding to $k[x]_>$-module generators of $\Hom_{k[x]_>/I}(J, J)$) we define on $k[t, x]$ a block ordering $(>_1, >)$ with $>_1$ a global ordering on the (first) $t$-block (i.e. $t_i > 1$ for all $i$ and $t_i > x_j$ for all $i, j$) and $>$ the given ordering on the (second) $x$-block. Then we continue with this new ring and monomial ordering.

This algorithm is correct (by applying Lemma \ref{lem:abnorm} to $A = k[x]_>/I$) and terminates because $\overline{k[x]_>/I}$ is finitely generated over $k[x]_>/I$ by Lemma \ref{lem:fingen}.

The second method is to compute the normalization of $k[x]/I'$ as in the previous section, with all variables greater than $1$. Then we map the result to $k[t, x]_{>_1, >}$ with block ordering $(>_1, >)$ as for the first method. By Lemma \ref{lem:fingen} both methods give the same result, hence the second algorithm is also correct.

If we start with an equidimensional decomposition $I' = \bigcap_{i=1}^r I_i$, then of course we only need to compute the normalization for those ideals $I_i$ for which a standard basis of $I_i$ w.r.t.~the ordering $>$ does not contain $1$.
\end{rem}

\begin{exmp}
To see the difference between both methods, let
$$I = \langle y^2 - x^2(x+1)^2(x+2) \rangle \subset R := k[x, y]_>,$$
with $>$ a local ordering (i.e. $k[x, y]_> \cong k[x, y]_{\langle x, y \rangle}$). Let $I' = I \cap k[x, y]$. In Figure \ref{fig:curve} we can see the real part of the curve $\V(I')$. This curve has two singularities, at the points $P_1 = (0, 0)$ and $P_2 = (-1, 0)$.

\begin{figure}[ht]
\begin{center}
\includegraphics[scale=0.6]{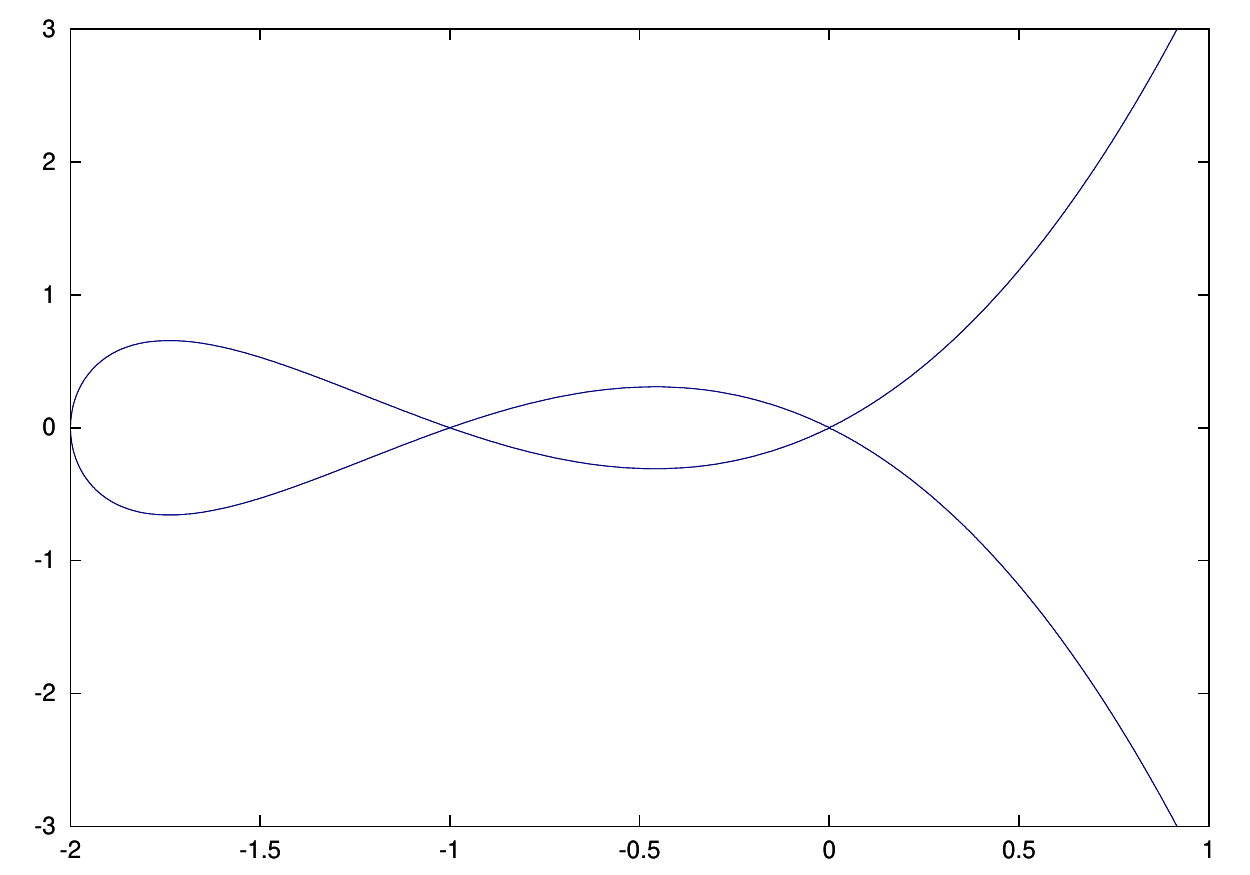}
\end{center}
\caption{$y^2 - x^2(x+1)^2(x+2)$}
\label{fig:curve}
\end{figure}

We carry out the first method, setting $A = R/I$. The singular locus of $I$ is $J = \langle x, y\rangle$, which is radical. This is the first test ideal. We take as non-zerodivisor $p := y$ and compute the quotient
$$U_1 := yJ :_{A} J = \langle x, y\rangle.$$
Since $U_1 \neq \langle y \rangle$ we go on.
The ring structure of $1/y \cdot U_1$ is $A_1 = k[t, x, y]_{>_1, >} / I_1$, with block ordering $(>_1, >)$ ($>_1$ any ordering) and
$I_1 = \langle
tx^4+4tx^3+5tx^2+2tx-y,
-ty+x,
t^2(x+1)^2(x+2)-1,
x^5+4x^4+5x^3+2x^2-y^2
\rangle
$.

We compute $J_1 := \sqrt{\varphi_1(\langle x, y\rangle)} = \langle x, y, 2t^2-1\rangle_{A_1}$.

Mapping $J_1$ to $Q(A)$ using $d_1 = y$ as denominator, we get $J_1 \cong 1/y \cdot H_1$, with $H_1 := \langle yx, y^2 \rangle$. (The image of $2t^2-1$ in $Q(A)$ is $(-10xy-8x^2y-2x^3y)/y$, which is already in $1/y \cdot \langle yx, y^2 \rangle$.)
We compute the quotient
$$U_2 := y^2 \langle yx, y^2 \rangle :_A \langle yx, y^2 \rangle = \langle xy, y^2\rangle.$$
We see that $y U_1 = U_2$. This means that $A_1$ was already normal and isomorphic to the normalization of $A$, which is therefore $1/y \cdot \langle x, y\rangle_{A}$.\\[1ex]
Let us now apply the second method. We set $R' := k[x, y]$ and $A' = R' / I'$. The singular locus of $I'$ is $J = \langle x^2 + x, y\rangle$, which is radical. $J$ serves as first test ideal. As non-zerodivisor we choose $p := y$ and compute the quotient
$$U_1 := yJ :_{A'} J = \langle y, x^3 + 3x^2 + 2x\rangle.$$
As $U_1 \neq \langle y \rangle$, we continue. We compute $A'_1$, the ring structure of $1/y \cdot U_1$,
$A'_1 = k[t, x, y]/\langle tx^2+tx-y, -ty+x^3+3x^2+2x, t^2-x-2, x^5+4x^4+5x^3+2x^2-y^2 \rangle,$
and $J_1 = \sqrt{\varphi_1(\langle x^2 + x, y\rangle)} = \langle x^2 + x, y\rangle$.

Mapping $J_1$ to $Q(A')$ using $d_1 = y$ as denominator, we obtain $J_1 \cong 1/y \cdot H_1$, with $H_1 := \langle y(x^2 + x), y^2 \rangle$.
We compute the quotient
$$U_2 := y^2 \langle y(x^2 + x), y^2 \rangle :_{A'} \langle y(x^2 + x), y^2 \rangle = \langle y^2, y(x^3 + 3x^2 + 2x)\rangle.$$
Now we have $y U_1 = U_2$, and thus $A'_1$ was already normal and isomorphic to the normalization of $A'$. Therefore, the normalization $\bar{A}$ equals $1/y \cdot \langle y, x^3 + 3x^2 + 2x\rangle_{A} = 1/y \cdot \langle y, x \rangle_{A}$, as before.
\end{exmp}

\begin{rem}
In the previous example, using the first method yields simpler test ideals and quotients. However, our experience is that in general, computations with non-global orderings are often slower than computations with global orderings, and therefore the second method should be preferred at least if the input ideal is prime. On the other hand the computation should be faster with the first method if the ideal, or its jacobian ideal, has complicated components which vanish in the localization.
\end{rem}

{\bf Acknowledgements.} 
This work was done while the second author was visiting the University of Kaiserslautern in February 2009; he is very grateful to the \singular\ team for their kind hospitality. We also like to thank the referees for their comments which helped to improve the presentation of the results.

\bibliographystyle{elsart-harv}

\end{document}